\documentclass{amsart}
\usepackage{amssymb,amscd,graphics,axodraw}
\usepackage[enableskew]{youngtab}
\usepackage[usenames]{color}
\usepackage[all]{xy}

\numberwithin{equation}{section}

\newtheorem{prop}{Proposition}
\newtheorem{theorem}[prop]{Theorem}

\newtheorem{lemma}[prop]{Lemma}

\theoremstyle{definition}

\newtheorem{example}[prop]{Example}

\numberwithin{prop}{section}

\newcommand{\El}{\mbox{\boldmath $L$}}
\newcommand{\Els}{\mbox{\scriptsize\boldmath $L$}}
\newcommand{\geh}{\mathfrak{g}}
\newcommand{\gnode}{
	\color[cmyk]{0,0,0,0.3}
	\put(-0.3,0){\rule{9.8pt}{9.8pt}}
	\color{black}}

\newcommand{\Jt}{\tilde{J}}
\newcommand{\la}{\lambda}
\newcommand{\La}{\Lambda}
\newcommand{\Lab}{\overline{\Lambda}}

\newcommand{\maru}[1]{\xymatrix@1{*+[o][F-]{#1}}}
\newcommand{\nut}{\tilde{\nu}}
\newcommand{\ol}{\overline}
\newcommand{\ot}{\otimes}
\newcommand{\Path}{\mathcal{P}}
\newcommand{\qbin}[2]{\genfrac{[}{]}{0pt}{}{#1}{#2}}
\newcommand{\RC}{\mathrm{RC}}
\newcommand{\veps}{\varepsilon}
\newcommand{\vphi}{\varphi}
\newcommand{\wt}{\mathrm{wt}}
\newcommand{\Z}{\mathbb{Z}}

\begin{document}

\title{KKR type bijection for the exceptional affine algebra $E_6^{(1)}$}

\author[M.~Okado]{Masato Okado}
\address{Department of Mathematical Science,
Graduate School of Engineering Science, Osaka University,
Toyonaka, Osaka 560-8531, Japan}

\author[N.~Sano]{Nobumasa Sano}

\date{May 20, 2011}
\footnote[0]{2010 {\it Mathematics Subject Classification.} 
Primary 17B37 82B23 05A19; Secondary 17B25 81R50 81R10 05E10 11B65.}

\begin{abstract}
For the exceptional affine type $E_6^{(1)}$ we establish a statistic-preserving bijection
between the highest weight paths consisting of the simplest Kirillov-Reshetikhin crystal
and the rigged configurations. The algorithm only uses the structure of the crystal graph,
hence could also be applied to other exceptional types.
\end{abstract}

\maketitle


\section{Introduction}

In a pioneering work \cite{KKR} Kerov, Kirillov and Reshetikhin introduced a new combinatorial
object, called rigged configuration, through Bethe ansatz analysis of the Heisenberg spin chain,
and constructed a bijection between rigged configurations and semistandard tableaux. One of the
amazing properties of the rigged configuration is that it possesses a natural statistic
and the statistic coincides with the charge by Lascoux and Sch\"utzenberger \cite{LS} on the tableau
under the bijection. Subsequently, Nakayashiki and Yamada \cite{NY} studied the meaning of the
charge in terms of Kashiwara's crystal bases. They considered the crystal base $B_l$ of the 
$l$-fold symmetric tensor representation of the $n$-dimensional irreducible 
$U_q(\widehat{\mathfrak{sl}}_n)$-module. For the tensor product $B_l\ot B_{l'}$ an integer-valued 
function $H$, called energy function, is defined via the $q\to0$ limit of the quantum $R$-matrix.
Using this $H$ they constructed a function $D$ on the multiple tensor product 
$B_{l_1}\ot\cdots\ot B_{l_m}$. They then showed that under a certain bijection sending highest 
weight vectors or paths of $B_{l_1}\ot\cdots\ot B_{l_m}$ to semistandard tableaux, the value of
$D$ agrees with the charge, thereby proving that the well-known Kostka polynomial is represented as
a generating function of highest weight paths with statistic $D$. This generating function is 
denoted by $X$ and the one of rigged configurations by $M$. The equality $X=M$ 
was extended to the most general case for affine type $A$ in \cite{KSS}. See also \cite{Srev} for review.

It did not take long before this kind of equality was conjectured to exist for other affine types.
For the $X$ side, crystal bases for some finite-dimensional modules, which are now called 
Kirillov-Reshetikhin (KR) modules, for quantum affine algebras have been discovered in \cite{KMN2}.
For the $M$ side, the existence of KR modules were conjectured and a formula to count the number
of rigged configurations were presented in \cite{KR}. Introducing an appropriate $q$-analogue for
the formula, the $X=M$ conjecture \cite{HKOTY,HKOTT} was presented. Imitating the one by KKR
a bijection between rigged configurations and highest weight paths consisting of elements of 
KR crystals for other nonexceptional affine types was subsequently constructed in \cite{OSS,S,SS}.
We note that these bijections have an important application for the analysis of the ultra-discrete
integrable systems, also called box-ball systems \cite{FOY,HHIKTT,HKOTY2}. In such systems 
rigged configurations give the complete set of the action and angle variables \cite{KOSTY,KSY}.

In this paper we consider the exceptional affine algebra of type $E_6^{(1)}$. The KR crystal we
deal with is the simplest one denoted in our notation by $B^{1,1}$, whose crystal structure was 
revealed in \cite{NS,JS}. We construct a map $\Phi$ from rigged configurations to highest weight 
elements of $(B^{1,1})^{\ot L}$ by executing a fundamental procedure $\delta$ repeatedly. We then
show $\Phi$ is a statistic-preserving bijection (Theorem \ref{thm:bij}). It is worth mentioning
that our procedure only uses the crystal graph structure of the KR crystal $B^{1,1}$, hence similar
constructions could be possible for other exceptional types.

We remark that recently Naoi \cite{N} solved, with the help of the results in \cite{DK} and \cite{NS2},
the $X=M$ conjecture for all untwisted affine types when the tensor product of KR crystals is of
the form $B^{r_1,1}\ot\cdots\ot B^{r_l,1}$ by showing both $X$ and $M$ are equal to the graded 
character of a Weyl module, a finite-dimensional current algebra representation defined in \cite{CL}.
Hence his result includes ours as a special case. However, we think our direct method is also important, 
since it could also be used for more general cases by cutting larger KR crystals as in \cite{KSS}.

\section{Quantum affine algebra and crystal}

\subsection{Affine algebra $E_6^{(1)}$} \label{subsec:algebra}
We consider in this paper the exceptional affine algebra $E_6^{(1)}$. 
The Dynkin diagram is depicted in Figure \ref{fig:Dynkin}. Note that we follow \cite{Kac} for 
the labeling of the Dynkin nodes. It is different from that in \cite{Bour} or \cite{JS}.
Let $I$ be the index set of the Dynkin nodes, and let $\alpha_i,\alpha_i^\vee,\La_i\,(i\in I)$ be 
simple roots, simple coroots, fundamental weights, respectively. Following the notation in \cite{Kac}
we denote the projection of $\La_i$ onto the weight space of $E_6$ by $\Lab_i$ ($i\in I_0$) and set
$\ol{P}=\bigoplus_{i\in I_0} \Z\Lab_i,\ol{P}^+=\bigoplus_{i\in I_0} \Z_{\ge0}\Lab_i$. Let 
$(C_{ij})_{i,j\in I}$ stand for the Cartan matrix for $E_6^{(1)}$. For $i,j\in I$, $i\sim j$ means
$C_{ij}=-1$, namely, the nodes $i$ and $j$ are adjacent in the Dynkin diagram of $E_6^{(1)}$.
{\unitlength=.95pt
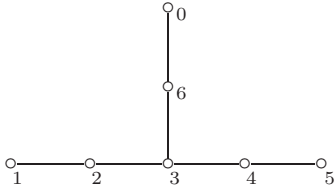
\begin{figure}
\[
\xymatrix{
&& *{\circ}<3pt> \ar@{-}[d]^<{0} \\
&& *{\circ}<3pt> \ar@{-}[d]^<{6} \\
*{\circ}<3pt> \ar@{-}[r]_<{1} & *{\circ}<3pt> \ar@{-}[r]_<{2} & *{\circ}<3pt> \ar@{-}[r]_<{3} & 
*{\circ}<3pt> \ar@{-}[r]_<{4}_(1.11){5} & *{\circ}<3pt>
}
\]
\caption{Dynkin diagram for $E_6^{(1)}$}
\label{fig:Dynkin}
\end{figure}}

\subsection{KR crystal}
Let $\geh$ be any affine algebra and $U'_q(\geh)$ the corresponding quantized enveloping algebra 
without the degree operator. Among finite-dimensional $U'_q(\geh)$-modules there is a distinguished
family called Kirillov-Reshetikhin (KR) modules \cite{KNT,Nak,Her}. One of the remarkable properties 
of KR modules is the existence of a crystal basis \cite{Ka} called a KR crystal. It was conjectured 
in \cite{HKOTY,HKOTT}, and recently settled for all nonexceptional types in \cite{OS}. The KR crystal 
is indexed by $(a,i)$ ($a\in I_0,i\in\Z_{>0}$) and denoted by $B^{a,i}$. For exceptional types the KR
crystal is known to exist when the KR module is irreducible or the index $a$ is adjacent to $0$ 
\cite{KMN2}. Recently, the explicit crystal structure of all such cases of type $E_6^{(1)}$ was
clarified in \cite{JS}.

The KR crystal we are interested in in this paper is an $E_6^{(1)}$-crystal $B^{1,1}$, whose crystal
structure was clarified in \cite{JS}. The crystal structure of $B^{1,1}$ is depicted
in Figure \ref{fig:crystal}.
\begin{figure}
\[
\entrymodifiers={+[o][F-]}
\xymatrix{
17 \ar@{.>}[d]^0 & 22 \ar@{.>}[d]^0 & 24 \ar@{.>}[d]^0 & 25 \ar@{.>}[d]^0 &
26 \ar@{.>}[d]^0 & 27 \ar@{.>}[d]^0 &*{}&*{}&*{}&*{} \\
1 \ar@{->}[r]^1 & 2 \ar@{->}[r]^2 & 3 \ar@{->}[r]^3 & 4 \ar@{->}[r]^4 \ar@{->}[d]^6 &
5 \ar@{->}[r]^5 \ar@{->}[d]^6 & 6 \ar@{->}[d]^6 &*{}&*{}&*{}&*{} \\
*{}&*{}&*{}& 7 \ar@{->}[r]^4 & 8 \ar@{->}[r]^5 \ar@{->}[d]^3 & 9 \ar@{->}[d]^3 &*{}&*{}&*{}&*{} \\
*{}&*{}&*{}&*{}& 10 \ar@{->}[r]^5 \ar@{->}[d]^2 & 11 \ar@{->}[r]^4 \ar@{->}[d]^2 & 12 \ar@{->}[d]^2 &
*{}&*{}&*{} \\
*{}&*{}&*{}&*{}& 13 \ar@{->}[r]^5 \ar@{->}[d]^1 & 14 \ar@{->}[r]^4 \ar@{->}[d]^1 &
15 \ar@{->}[r]^3 \ar@{->}[d]^1 & 16 \ar@{->}[r]^6 \ar@{->}[d]^1 & 17 \ar@{.>}[r]^0 \ar@{->}[d]^1 &*{} \\
*{}&*{}&*{}&*{}& 18 \ar@{->}[r]^5 & 19 \ar@{->}[r]^4 & 20 \ar@{->}[r]^3 & 21 \ar@{->}[r]^6 \ar@{->}[d]^2 &
22 \ar@{.>}[r]^0 \ar@{->}[d]^2 &*{} \\
*{}&*{}&*{}&*{}&*{}&*{}&*{}& 23 \ar@{->}[r]^6 & 24 \ar@{.>}[r]^0 \ar@{->}[d]^3 &*{} \\
*{}&*{}&*{}&*{}&*{}&*{}&*{}&*{}& 25 \ar@{->}[d]^4 \ar@{.>}[r]^0 &*{} \\
*{}&*{}&*{}&*{}&*{}&*{}&*{}&*{}& 26 \ar@{->}[d]^5 \ar@{.>}[r]^0 &*{} \\
*{}&*{}&*{}&*{}&*{}&*{}&*{}&*{}& 27 \ar@{.>}[r]^0 &*{}
}
\]
\caption{Crystal graph for $B^{1,1}$}
\label{fig:crystal}
\end{figure}
Here vertices in the graph signify elements of $B^{1,1}$ and $b\stackrel{i}{\longrightarrow}b'$
stands for $f_ib=b'$ or equivalently $b=e_ib'$. We adopt the original convention for the tensor
product of crystals. Namely, if $B_1$ and $B_2$ are crystals, then for $b_1\otimes b_2\in B_1\otimes B_2$
the action of $e_{i}$ is defined as
\begin{equation*}
e_{i}(b_1\otimes b_2)=\begin{cases}
e_{i}b_1 \otimes b_2 &\text{if $\vphi_i(b_1)\ge\veps_i(b_2)$,}\\
b_1\otimes e_{i} b_2 &\text{else,}
\end{cases}
\end{equation*}
where $\veps_i(b)=\max\{k\mid e_i^kb\ne0\}$ and
$\vphi_i(b)=\max\{k\mid f_i^kb\ne0\}$.

By glancing at Figure \ref{fig:crystal}, one obtains the following lemma which will be used
to prove our main theorem. Let $B_0$ be the subgraph obtained by ignoring the 0-arrows from $B$.
A route is a sequence $(\beta_1,\ldots,\beta_l)$ of arrows such that the sink of $\beta_j$ is the
source of $\beta_{j+1}$ for $j=1,\ldots,l-1$.

\begin{lemma} \label{lem:graph}
The graph $B_0$ has the following features.
\begin{enumerate}
\item Suppose the initial arrow of a route $R$ has the same color $a$ as the terminal arrow 
	and there is no intermidiate arrow of color $a$. Then there are exactly two arrows $\beta_i$
	$(i=1,2)$ of color $b_i$ such that $b_i\sim a$ in $R$.
\item Let $R$ be a route starting from $\maru{1}$, $(a_1,\ldots,a_l)$ the colors from the 
	initial arrow to the terminal one in $R$. Then we have
\[
\sum_{j=1}^{l-1}C_{a_ja_l}=\delta_{a_l,1}-1.
\]
\item Let $R$ be a route of two steps with colors $(a,b)$ such that $b\not\sim a$. Then there 
	exists a route $R'$ with colors $(b,a)$ starting and terminating at the same vertices 
	as $R$.
\item Let $R$ be a route of colors $(a_1,\ldots,a_l)$. Let $v_i$ be the source of the arrow of
	color $a_i$ $(i=1,\ldots,l)$. Suppose $a_1\sim a_l$ and $a_i\not\sim a_l$ for any 
	$i=2,\ldots,l-1$. Then there is an arrow of color $a_l$ starting from $v_i$ for any 
	$i=2,\ldots,l-1$.
\end{enumerate}
\end{lemma}

\begin{proof}
(1) and (3) can be checked by direct observations. (2) and (4) are derived from (1) and (3).
\end{proof}

In what follows in this paper we assume $B=B^{1,1}$.
The set of classically restricted paths in $B^{\otimes L}$ of weight
$\la\in\ol{P}^+$ is by definition
\begin{equation} \label{eq:cl hwv}
\Path(\la,L)=\{b\in B^{\otimes L}\mid \text{$\wt(b)=\la$ and $e_{i}b=0$
for all $i\in I_0$} \}.
\end{equation}
One may check that the
following are equivalent for $b=b_1\otimes b_2\otimes\dotsm\otimes b_L\in B^{\otimes L}$ and
$\la\in\ol{P}^+$.
\begin{enumerate}
\item $b$ is a classically restricted path of weight $\la\in\ol{P}^+$.
\item $b_1\otimes\dotsm\otimes b_{L-1}$ is a classically restricted path of weight 
	$\la-\wt(b_L)$, and $\veps_i(b_L)\le\langle\la-\wt(b_L),\alpha^\vee_i\rangle$ for all $i\in I_0$.
\end{enumerate}

The weight function $\wt:B\to\ol{P}$ is given by 
$\wt(b)=\sum_{i\in I}(\vphi_i(b)-\veps_i(b))\La_i$.
The weight function $\wt:B^{\otimes L}\to\ol{P}$ is defined by
$\wt(b_1\otimes\dotsm\otimes b_L) =\sum_{j=1}^L \wt(b_j)$.

\begin{example} \label{ex:typical}
The element 
\[
b=\maru{1}\cdot\maru{2}\cdot\maru{3}\cdot\maru{16}\cdot\maru{2}\cdot\maru{24}
\]
of $B^{\otimes6}$ is a classically restricted path of weight $\Lab_3$. The dot $\cdot$ signifies 
$\otimes$.
\end{example}

\subsection{One-dimensional sums} \label{subsec:1dsum}
The energy function
$D:B^{\otimes L}\rightarrow\Z$ gives the grading on $B^{\otimes L}$. In
our case where a path is an element of the tensor product of a single KR crystal it takes a simple form. 
Due to the existence of the universal $R$-matrix and the fact that
$B\otimes B$ is connected, by \cite{KMN1} there is a unique (up
to global additive constant) function $H:B\otimes
B\rightarrow\Z$ called the local energy function, such that
\begin{equation} \label{eq:loc en}
  H(e_{i}(b\otimes b')) = H(b\otimes b')+
  \begin{cases}
  1 & \text{if $i=0$ and $e_{0}(b\otimes b')=e_{0}b\otimes b'$}
  \\
  -1 & \text{if $i=0$ and $e_{0}(b\otimes b')=b\otimes e_{0}b'$}
  \\
  0 & \text{otherwise.}
  \end{cases}
\end{equation}
We normalize $H$ by the condition
\begin{equation} \label{eq:H norm}
  H(\maru{1} \otimes \maru{1}) = 0.
\end{equation}

More specifically, the value of $H$ is calculated as follows. Firstly, one knows the crystal
graph of $B_0\otimes B_0$ decomposes into three connected components as
\[
B_0\otimes B_0=B(2\Lab_1)\oplus B(\Lab_1+\Lab_2)\oplus B(\Lab_1+\Lab_5),
\]
where $B(\la)$ stands for the highest weight $E_6$-crystal of highest weight $\la$ and the
highest weight vector of each component is given by $\maru{1}\otimes\maru{1},\maru{1}\otimes\maru{2},
\maru{1}\otimes\maru{18}$. $H$ is constant on each component, and takes the value $0,-1,-2$, 
respectively. One can confirm it from the fact that $e_0(\maru{1}\otimes\maru{1})=
\maru{1}\otimes\maru{17}$ and $e_0(\maru{1}\otimes\maru{2})=\maru{1}\otimes\maru{22}$ belong
to the second and third component.

With this $H$ the energy function $D$ is defined by
\begin{equation} \label{eq:energy}
D(b_1\otimes\dotsm \otimes b_L) = 
\sum_{j=1}^{L-1} (L-j) \,\,H(b_j\otimes b_{j+1}).
\end{equation}
Define the one-dimensional sum $X(\la,L;q)\in\Z_{\ge0}[q^{-1}]$ by
\begin{equation} \label{eq:onedim}
  X(\la,L;q) = \sum_{b\in \Path(\la,L)} q^{D(b)}.
\end{equation}

\section{Rigged configuration and the bijection}
\label{sec:bijection}
\subsection{The fermionic formula} \label{subsec:ferm}
This subsection reviews the definition of the fermionic formula from
\cite{HKOTT,HKOTY}. We at first provide the definition that is valid for any 
simply-laced affine type $\geh$ and datum $\El$, and then restrict $\geh$ and $\El$ to $E_6^{(1)}$
and the case corresponding to paths we consider in this paper. Fix $\la\in\ol{P}^+$ and a matrix
$\El=(L_i^{(a)})_{a\in I_0,i\in\Z_{>0}}$ of nonnegative integers, almost all zero.
Let $\nu=(m_i^{(a)})$ be another such matrix. Say that $\nu$ is an admissible configuration if
it satisfies
\begin{equation}
\label{eq:config} \sum_{\substack{a\in I_0 \\ i\in\Z_{>0}}} i\,
m_i^{(a)} \alpha_a = \sum_{\substack{a\in I_0 \\
i\in\Z_{>0}}} i \,L_i^{(a)} \Lab_a - \la
\end{equation}
and
\begin{equation} \label{eq:ppos}
  p_i^{(a)} \ge 0\qquad\text{for all $a\in I_0$ and
  $i\in\Z_{>0}$,}
\end{equation}
where
\begin{equation} \label{eq:p}
p_i^{(a)} = \sum_{j\in\Z_{>0}} \left( L_j^{(a)} \min(i,j) -
\sum_{b\in I_0} (\alpha_a|\alpha_b)\min(i,j)m_j^{(b)}\right).
\end{equation}
Write $C(\la,\El)$ for the set of admissible configurations for $\la\in\ol{P}^+$ and $\El$. Define
the charge of a configuration $\nu$ by
\begin{equation} \label{eq:c}
\begin{split}
c(\nu) = &\dfrac{1}{2} \sum_{a,b\in I_0} \sum_{j,k\in\Z_{>0}} (\alpha_a|\alpha_b)
\min(j,k) m_j^{(a)} m_k^{(b)} \\
&\hspace{1cm}- \sum_{a\in I_0} \sum_{j,k\in\Z_{>0}} \min(j,k) L_j^{(a)} m_k^{(a)}.
\end{split}
\end{equation}
Using \eqref{eq:p} $c(\nu)$ is rewritten as
\begin{equation} \label{rewrite c}
c(\nu)=-\frac12\left(\sum_{a\in I_0,i\in\Z_{>0}}p^{(a)}_im^{(a)}_i+
\sum_{a\in I_0,j,k\in\Z_{>0}}\min(j,k)L^{(a)}_jm^{(a)}_k\right).
\end{equation}
The fermionic formula is then defined by
\begin{equation} \label{fermi}
M(\la,\El;q) = \sum_{\nu\in C(\la,\Els)} q^{c(\nu)}
\prod_{a\in I_0} \prod_{i\in\Z_{>0}}
\qbin{\,p_i^{(a)}+m_i^{(a)}}{m_i^{(a)}}.
\end{equation}

We now set $\geh=E_6^{(1)}$ and
\begin{equation} \label{L^(a)_i}
L^{(a)}_i=L\delta_{a1}\delta_{i1}\qquad(a\in I_0,i\in\Z_{>0}).
\end{equation}
The latter restriction corresponds to considering paths in $(B^{1,1})^{\otimes L}$. By abuse of
notation we denote the fermionic formula under the restriction \eqref{L^(a)_i} by $M(\la,L;q)$.
Then the $X=M$ conjecture of \cite{HKOTY,HKOTT} states in this particular case that
\begin{equation}\label{eq:X=M}
  X(\la,L;q)=M(\la,L;q).
\end{equation}

\subsection{Rigged configuration}
The fermionic formula $M(\la,\El;q)$ can be interpreted using
combinatorial objects called rigged configurations. These objects
are a direct combinatorialization of the fermionic formula
$M(\la,\El;q)$. Our goal is to prove \eqref{eq:X=M} by defining a statistic-preserving
bijection from rigged configurations to classically restricted paths.
Let $\nu=(m^{(a)}_i)_{a\in I_0,i\in\Z_{>0}}$ be an admissible configuration. We identify $\nu$
with a sequence of partitions $\{\nu^{(a)}\}_{a\in I_0}$ such that $\nu^{(a)}=(1^{m^{(a)}_1}
2^{m^{(a)}_2}\cdots)$. Let $J=\{J^{(a,i)}\}_{(a,i)\in I_0\times\Z_{>0}}$ be a double sequence of partitions.
Then a rigged configuration is a pair $(\nu,J)$ subject to the
restriction \eqref{eq:config} and the requirement that $J^{(a,i)}$
be a partition contained in a $m_i^{(a)} \times p_i^{(a)}$ rectangle.

For a partition $\mu$ and
$i\in\Z_{>0}$, define
\begin{equation} \label{eq:Qdef}
Q_i(\mu)=\sum_j \min(\mu_j,i),
\end{equation}
the area of $\mu$ in the first $i$ columns. Then setting $Q^{(a)}_i=Q_i(\nu^{(a)})$
the vacancy number \eqref{eq:p} under the restriction \eqref{L^(a)_i} is rewritten as

\begin{equation} \label{p and Q}
p^{(a)}_i=L\delta_{a1}-2Q^{(a)}_i+\sum_{b\sim a}Q^{(b)}_i,
\end{equation}
where $b\sim a$ stands for $C_{ba}=-1$ as defined in \S\ref{subsec:algebra}.

The set of rigged configurations for
fixed $\la$ and $\El$ is denoted by $\RC(\la,\El)$. Then
\eqref{fermi} is equivalent to
\begin{equation*}
M(\la,\El;q)=\sum_{(\nu,J)\in\RC(\la,\Els)} q^{c(\nu,J)}
\end{equation*}
where 
\begin{equation} \label{c-nu-J}
c(\nu,J)=c(\nu)+|J|
\end{equation}
with $c(\nu)$ as in \eqref{eq:c}
and $|J|=\sum_{(a,i)\in I_0\times\Z_{>0}}|J^{(a,i)}|$.
The set $\RC(\la,\El)$ with the restriction \eqref{L^(a)_i} is denoted by $\RC(\la,L)$.

\begin{example} \label{ex:RC}
A rigged configuration in $\RC(\ol{\La}_3,6)$ is illustrated below.
\medskip

\begin{center}
\unitlength 10pt
\begin{picture}(34,5)
\Yboxdim{10pt}
\put(1,0){\yng(2,1,1,1,1)}
\put(0.2,3){1}
\put(0.2,4){0}
\put(2.2,0){0}
\put(2.2,1){0}
\put(2.2,2){1}
\put(2.2,3){1}
\put(3.2,4){0}
\put(7,0){\yng(2,1,1,1,1)}
\put(6.2,3){0}
\put(6.2,4){0}
\put(8.2,0){0}
\put(8.2,1){0}
\put(8.2,2){0}
\put(8.2,3){0}
\put(9.2,4){0}
\put(13,0){\yng(2,1,1,1,1)}
\put(12.2,3){0}
\put(12.2,4){1}
\put(14.2,0){0}
\put(14.2,1){0}
\put(14.2,2){0}
\put(14.2,3){0}
\put(15.2,4){1}
\put(19,2){\yng(2,1,1)}
\put(18.2,3){0}
\put(18.2,4){0}
\put(20.2,2){0}
\put(20.2,3){0}
\put(21.2,4){0}
\put(25,4){\yng(2)}
\put(24.2,4){0}
\put(27.2,4){0}
\put(31,3){\yng(2,1)}
\put(30.2,3){1}
\put(30.2,4){0}
\put(32.2,3){1}
\put(33.2,4){0}
\end{picture}
\end{center}
\medskip
The partitions $\nu^{(1)},\nu^{(2)},\ldots,\nu^{(6)}$ are illustrated from left to right
as Young diagrams. In $\nu^{(1)}$, 0 and 1 on the left signify $p^{(1)}_2=0$ and $p^{(1)}_1=1$.
Looking on the right we see $J^{(1)}_2=(0),J^{(1)}_1=(1,1,0,0)$. From \eqref{rewrite c} we
have $c(\nu)=-18$, hence $c(\nu,J)=-14$.
\end{example}

\subsection{The bijection from RCs to paths}
We now describe the bijection
$\Phi:\RC(\la,L)\to\Path(\la,L)$.
Let $(\nu,J)\in\RC(\la,L)$. We shall define a map
$\gamma:\RC(\la,L)\to B$ which associates to $(\nu,J)$ an element of $B$.
Denote by $\RC_b(\la,L)$ the elements of $\RC(\la,L)$ such that $\gamma(\nu,J)=b$. 
We shall define a bijection
$\delta:\RC_b(\la,L)\to\RC(\la-\wt(b),L-1)$. The disjoint union
of these bijections then defines a bijection
$\delta:\RC(\la,L)\to\bigsqcup_{b\in B} \RC(\la-\wt(b),L-1)$.

The bijection $\Phi$ is defined recursively as follows. For $b\in
B$ let $\Path_b(\la,L)$ be the set of paths in
$B^{\otimes L}$ that have $b$ as rightmost tensor factor.
For $L=0$ the bijection $\Phi$ sends the empty rigged
configuration (the only element of the set $\RC(\la,L)$) to the
empty path (the only element of $\Path(\la,L)$). Otherwise
assume that $\Phi$ has been defined for $B^{\otimes (L-1)}$ and
define it for $B^{\otimes L}$ by the commutative diagram
\begin{equation}
\begin{CD}
\RC_b(\la,L) @>{\Phi}>> \Path_b(\la,L) \\
@V{\delta}VV @VVV \\
\RC(\la-\wt(b),L-1) @>{\Phi}>> \Path(\la-\wt(b),L-1)
\end{CD}
\end{equation}
where the right hand vertical map removes the rightmost tensor
factor $b$. In short,
\begin{equation}
  \Phi(\nu,J)=\Phi(\delta(\nu,J))\otimes \gamma(\nu,J).
\end{equation}

Here follows the main theorem of our paper.

\begin{theorem}\label{thm:bij}
$\Phi:\RC(\la,L)\to\Path(\la,L)$ is a bijection such that
\begin{equation} \label{eq:c=D}
c(\nu,J)=D(\Phi(\nu,J))\qquad\text{for all
$(\nu,J)\in\RC(\la,L)$.}
\end{equation}
\end{theorem}

\section{The bijection}
In this section, for $(\nu,J)\in\RC(\la,L)$, an
algorithm is given which defines $b=\gamma(\nu,J)$, the new smaller
rigged configuration $(\nut,\Jt)=\delta(\nu,J)$ such that
$(\nut,\Jt)\in\RC(\rho,L-1)$ where $\rho=\la-\wt(b)$, and the
new vacancy numbers in terms of the old.

Illustrating a rigged configuration as in Example \ref{ex:RC} we call a row in 
$\nu^{(a)}$ \textit{singular} if its rigging (number on the right) is equal to the
corresponding vacancy number $p_i^{(a)}$.

\subsection{Algorithm $\delta$} \label{subsec:algorithm}
Suppose you are at $b=\maru{1}$ in the crystal graph $B_0$ and set $\ell_0=1$.
Repeat the following process for $j=1,2,\ldots$ until stopped. From $b$ proceed 
by one step through an arrow of color $a$. Find the minimal integer $i\ge\ell_{j-1}$ 
such that $\nu^{(a)}$ has a singular row of length $i$ and set $\ell_j=i$, reset $b$ to be the sink of
the arrow. If there is no such integer, then set $\ell_j=\infty$ and stop. If there are two arrows sourcing
from $b$, compare the minimal integers and take the smaller one. If the integers are 
the same, either one can be taken. The output of the algorithm does not depend on
the choices by Lemma \ref{lem:graph} (3).

We also use the notation $\ell_k^{(a)}(=\ell_j)$ if at the $j$-th step the arrow has
color $a$ and it is the $k$-th one having color $a$ from the beginning.

\subsection{New configuration} \label{subsec:rc}
The new configuration $\tilde{\nu}=(\tilde{m}^{(a)}_i)$ is changed to 
\begin{equation} \label{eq:ch m D1}
 \tilde{m}_i^{(a)}=m_i^{(a)}-
\sum_{k=1}^{k_a}(\delta_{i,\ell_k^{(a)}}-\delta_{i,\ell_k^{(a)}-1})
\end{equation}
where $k_a$ is the maximum of $k$ such that $\ell_k^{(a)}$ is finite.

\subsection{Change in vacancy numbers} \label{subsec:vacancy}
Let $A$ be a statement, then $\chi(A)=1$ if $A$ is true and $\chi(A)=0$ if $A$ is false.
Then from \eqref{p and Q} one has 
\[
\tilde{p}^{(1)}_i-p^{(1)}_i=-1 +2\chi(i\ge \ell^{(1)}_1)-\chi(i\ge \ell^{(2)}_1)-\chi(i\ge \ell^{(2)}_2)
+2\chi(i\ge \ell^{(1)}_2)-\chi(i\ge \ell^{(2)}_3).
\]
Here we set $\ell^{(a)}_k=\infty$ if $k>k_a$.
This calculation is summarized in the following table.

$\bullet\; a=1$\smallskip

\begin{tabular}[htb]{|c|c|c|c|c|c|} \hline
$[1,\ell^{(1)}_1)$
&$[\ell^{(1)}_1,\ell^{(2)}_1)$
&$[\ell^{(2)}_1,\ell^{(2)}_2)$
&$[\ell^{(2)}_2,\ell^{(1)}_2)$
&$[\ell^{(1)}_2,\ell^{(2)}_3)$
&$[\ell^{(2)}_3,\infty)$ \\ \hline
--1&+1&0&--1&+1&0 \\ \hline
\end{tabular}\\
The first row signifies the range of $i$, namely, $[1,\ell^{(1)}_1)$ means $1\le i<\ell^{(1)}_1$ and
the second row $\tilde{p}^{(1)}_i-p^{(1)}_i$ in this range.
Similarly one obtains the following tables for other $a$.

$\bullet\; a=2$\smallskip

\begin{tabular}[htb]{|c|c|c|c|c|c|} \hline
$[1,\ell^{(1)}_1)$
&$[\ell^{(1)}_1,\ell^{(2)}_1)$
&$[\ell^{(2)}_1,\ell^{(3)}_1)$
&$[\ell^{(3)}_1,\ell^{(3)}_2)$
&$[\ell^{(3)}_2,\ell^{(2)}_2)$
&$[\ell^{(2)}_2,\min(\ell^{(1)}_2,\ell^{(3)}_3))$ \\ \hline
0&--1&+1&0&--1&+1 \\ \hline
\end{tabular}\\

\begin{tabular}[htb]{|c|c|c|c|} \hline
$[\min,\max)$
&$[\max,\ell^{(2)}_3)$
&$[\ell^{(2)}_3,\ell^{(3)}_4)$
&$[\ell^{(3)}_4,\infty)$ \\ \hline
0&--1&+1&0 \\ \hline
\end{tabular}
\vspace{1mm}

\noindent
In this table $\min,\max$ without $(\cdot,\cdot)$ means the abbreviation of the
previous parenthesis.

$\bullet\; a=3$\smallskip

\begin{tabular}[htb]{|c|c|c|c|c|} \hline
$[1,\ell^{(2)}_1)$
&$[\ell^{(2)}_1,\ell^{(3)}_1)$
&$[\ell^{(3)}_1,\min(\ell^{(4)}_1,\ell^{(6)}_1))$
&$[\min,\max)$
&$[\max,\ell^{(3)}_2)$ \\ \hline
0&--1&+1&0&--1 \\ \hline
\end{tabular}\\

\begin{tabular}[htb]{|c|c|c|c|} \hline
$[\ell^{(3)}_2,\min(\ell^{(2)}_2,\ell^{(4)}_2))$
&$[\min,\max)$
&$[\max,\ell^{(3)}_3)$
&$[\ell^{(3)}_3,\min(\ell^{(6)}_2,\ell^{(2)}_3))$ \\ \hline
+1&0&--1&+1 \\ \hline
\end{tabular}\\

\begin{tabular}[htb]{|c|c|c|c|} \hline
$[\min,\max)$
&$[\max,\ell^{(3)}_4)$ 
&$[\ell^{(3)}_4,\ell^{(4)}_3)$
&$[\ell^{(4)}_3,\infty)$ \\ \hline
0&--1&+1&0 \\ \hline
\end{tabular}\\

$\bullet\; a=4$\smallskip

\begin{tabular}[htb]{|c|c|c|c|c|c|} \hline
$[1,\ell^{(3)}_1)$
&$[\ell^{(3)}_1,\ell^{(4)}_1)$
&$[\ell^{(4)}_1,\min(\ell^{(5)}_1,\ell^{(3)}_2))$
&$[\min,\max)$
&$[\max,\ell^{(4)}_2)$
&$[\ell^{(4)}_2,\ell^{(3)}_3)$ \\ \hline
0&--1&+1&0&--1&+1 \\ \hline
\end{tabular}\\

\begin{tabular}[htb]{|c|c|c|c|} \hline
$[\ell^{(3)}_3,\ell^{(3)}_4)$
&$[\ell^{(3)}_4,\ell^{(4)}_3)$
&$[\ell^{(4)}_3,\ell^{(5)}_2)$
&$[\ell^{(5)}_2,\infty)$ \\ \hline
0&--1&+1&0 \\ \hline
\end{tabular}\\

$\bullet\; a=5$\smallskip

\begin{tabular}[htb]{|c|c|c|c|c|c|} \hline
$[1,\ell^{(4)}_1)$
&$[\ell^{(4)}_1,\ell^{(5)}_1)$
&$[\ell^{(5)}_1,\ell^{(4)}_2) $
&$[\ell^{(4)}_2,\ell^{(4)}_3)$
&$[\ell^{(4)}_3,\ell^{(5)}_2)$
&$[\ell^{(5)}_2,\infty)$ \\ \hline
0&--1&+1&0&--1&+1 \\ \hline
\end{tabular}\\

$\bullet\; a=6$\smallskip

\begin{tabular}[htb]{|c|c|c|c|c|c|c|} \hline
$[1,\ell^{(3)}_1)$
&$[\ell^{(3)}_1,\ell^{(6)}_1)$
&$[\ell^{(6)}_1,\ell^{(3)}_2) $
&$[\ell^{(3)}_2,\ell^{(3)}_3)$
&$[\ell^{(3)}_3,\ell^{(6)}_2)$
&$[\ell^{(6)}_2,\ell^{(3)}_4)$
&$[\ell^{(3)}_4,\infty)$ \\ \hline
0&--1&+1&0&--1&+1&0 \\ \hline
\end{tabular}\\

\begin{example}
The algorithm $\Phi$ for the rigged configuration in Example \ref{ex:RC} is described at each
step $\delta$ below.

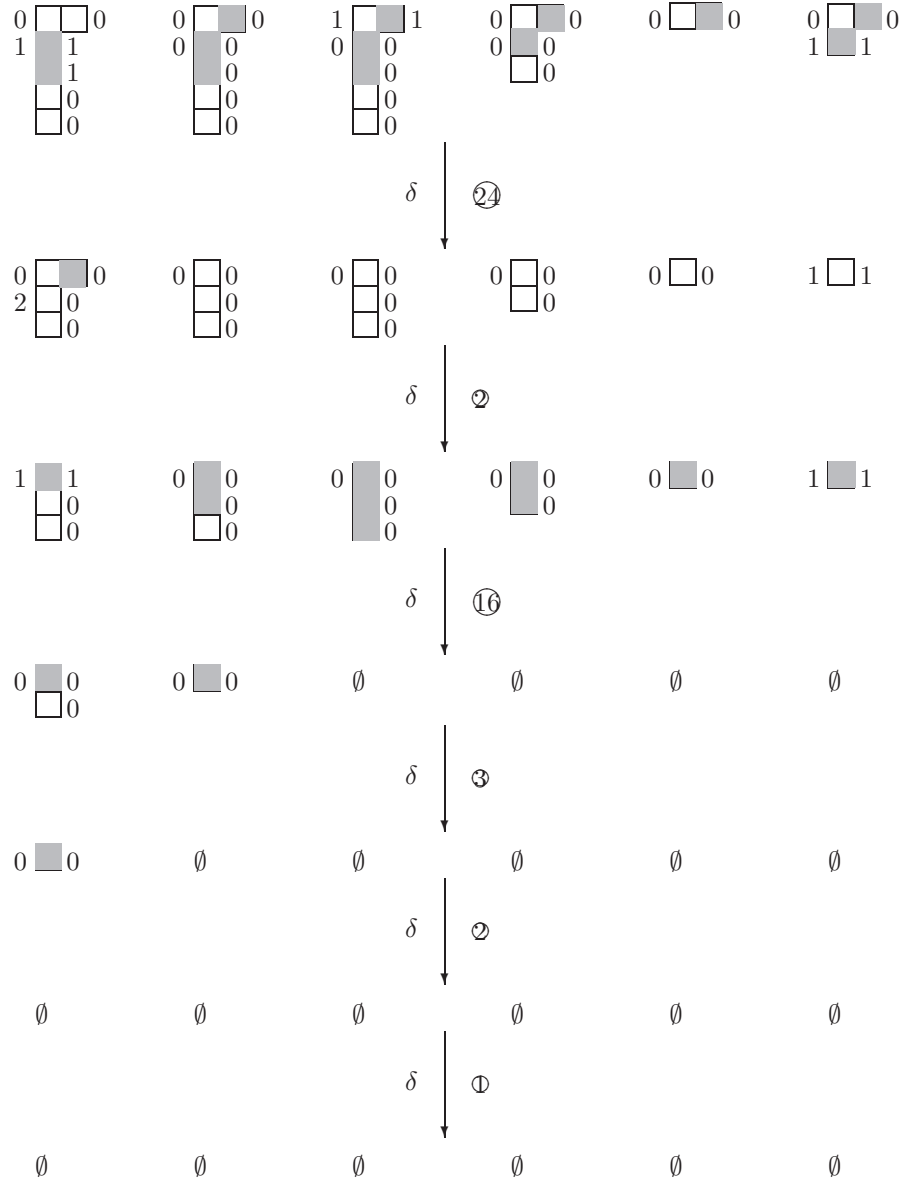
\begin{center}
\unitlength 10pt
\begin{picture}(34,5)
\Yboxdim{10pt}
\put(1,0){\yng(2,1,1,1,1)}
\put(1,1.9){\gnode}
\put(1,2.9){\gnode}
\put(0.2,3){1}
\put(0.2,4){0}
\put(2.2,0){0}
\put(2.2,1){0}
\put(2.2,2){1}
\put(2.2,3){1}
\put(3.2,4){0}
\put(7,0){\yng(2,1,1,1,1)}
\put(7,1.9){\gnode}
\put(7,2.9){\gnode}
\put(7.9,3.9){\gnode}
\put(6.2,3){0}
\put(6.2,4){0}
\put(8.2,0){0}
\put(8.2,1){0}
\put(8.2,2){0}
\put(8.2,3){0}
\put(9.2,4){0}
\put(13,0){\yng(2,1,1,1,1)}
\put(13,1.9){\gnode}
\put(13,2.9){\gnode}
\put(13.9,3.9){\gnode}
\put(12.2,3){0}
\put(12.2,4){1}
\put(14.2,0){0}
\put(14.2,1){0}
\put(14.2,2){0}
\put(14.2,3){0}
\put(15.2,4){1}
\put(19,2){\yng(2,1,1)}
\put(19,3){\gnode}
\put(20,3.9){\gnode}
\put(18.2,3){0}
\put(18.2,4){0}
\put(20.2,2){0}
\put(20.2,3){0}
\put(21.2,4){0}
\put(25,4){\yng(2)}
\put(26,3.95){\gnode}
\put(24.2,4){0}
\put(27.2,4){0}
\put(31,3){\yng(2,1)}
\put(31,3){\gnode}
\put(32,3.95){\gnode}
\put(30.2,3){1}
\put(30.2,4){0}
\put(32.2,3){1}
\put(33.2,4){0}
\end{picture}
\end{center}

\begin{center}
\unitlength 10pt
\begin{picture}(3,4)
\put(-0.5,1.8){$\delta$}
\put(2,2){\maru{24}}
\put(1,4){\vector(0,-1){4}}
\end{picture}
\end{center}

\begin{center}
\unitlength 10pt
\begin{picture}(34,3)
\Yboxdim{10pt}
\put(1,0){\yng(2,1,1)}
\put(1.9,1.9){\gnode}
\put(0.2,1){2}
\put(0.2,2){0}
\put(2.2,0){0}
\put(2.2,1){0}
\put(3.2,2){0}
\put(7,0){\yng(1,1,1)}
\put(6.2,2){0}
\put(8.2,0){0}
\put(8.2,1){0}
\put(8.2,2){0}
\put(13,0){\yng(1,1,1)}
\put(12.2,2){0}
\put(14.2,0){0}
\put(14.2,1){0}
\put(14.2,2){0}
\put(19,1){\yng(1,1)}
\put(18.2,2){0}
\put(20.2,1){0}
\put(20.2,2){0}
\put(25,2){\yng(1)}
\put(24.2,2){0}
\put(26.2,2){0}
\put(31,2){\yng(1)}
\put(30.2,2){1}
\put(32.2,2){1}
\end{picture}
\end{center}

\begin{center}
\unitlength 10pt
\begin{picture}(3,4)
\put(-0.5,1.8){$\delta$}
\put(2,2){\maru{2}}
\put(1,4){\vector(0,-1){4}}
\end{picture}
\end{center}

\begin{center}
\unitlength 10pt
\begin{picture}(34,3)
\Yboxdim{10pt}
\put(1,0){\yng(1,1,1)}
\put(1,1.9){\gnode}
\put(0.2,2){1}
\put(2.2,0){0}
\put(2.2,1){0}
\put(2.2,2){1}
\put(7,0){\yng(1,1,1)}
\put(7,1){\gnode}
\put(7,2){\gnode}
\put(6.2,2){0}
\put(8.2,0){0}
\put(8.2,1){0}
\put(8.2,2){0}
\put(13,0){\yng(1,1,1)}
\put(13,0){\gnode}
\put(13,1){\gnode}
\put(13,2){\gnode}
\put(12.2,2){0}
\put(14.2,0){0}
\put(14.2,1){0}
\put(14.2,2){0}
\put(19,1){\yng(1,1)}
\put(19,1){\gnode}
\put(19,2){\gnode}
\put(18.2,2){0}
\put(20.2,1){0}
\put(20.2,2){0}
\put(25,2){\yng(1)}
\put(25,2){\gnode}
\put(24.2,2){0}
\put(26.2,2){0}
\put(31,2){\yng(1)}
\put(31,2){\gnode}
\put(30.2,2){1}
\put(32.2,2){1}
\end{picture}
\end{center}

\begin{center}
\unitlength 10pt
\begin{picture}(3,4)
\put(-0.5,1.8){$\delta$}
\put(2,2){\maru{16}}
\put(1,4){\vector(0,-1){4}}
\end{picture}
\end{center}

\begin{center}
\unitlength 10pt
\begin{picture}(34,2)
\Yboxdim{10pt}
\put(1,0){\yng(1,1)}
\put(1,1){\gnode}
\put(0.2,1){0}
\put(2.2,0){0}
\put(2.2,1){0}
\put(7,1){\yng(1)}
\put(7,1){\gnode}
\put(6.2,1){0}
\put(8.2,1){0}
\put(13,1){$\emptyset$}
\put(19,1){$\emptyset$}
\put(25,1){$\emptyset$}
\put(31,1){$\emptyset$}
\end{picture}
\end{center}

\begin{center}
\unitlength 10pt
\begin{picture}(3,4)
\put(-0.5,1.8){$\delta$}
\put(2,2){\maru{3}}
\put(1,4){\vector(0,-1){4}}
\end{picture}
\end{center}

\begin{center}
\unitlength 10pt
\begin{picture}(34,1)
\Yboxdim{10pt}
\put(1,0){\yng(1)}
\put(1,0){\gnode}
\put(0.2,0){0}
\put(2.2,0){0}
\put(7,0){$\emptyset$}
\put(13,0){$\emptyset$}
\put(19,0){$\emptyset$}
\put(25,0){$\emptyset$}
\put(31,0){$\emptyset$}
\end{picture}
\end{center}

\begin{center}
\unitlength 10pt
\begin{picture}(3,4)
\put(-0.5,1.8){$\delta$}
\put(2,2){\maru{2}}
\put(1,4){\vector(0,-1){4}}
\end{picture}
\end{center}

\begin{center}
\unitlength 10pt
\begin{picture}(34,1)
\Yboxdim{10pt}
\put(1,0){$\emptyset$}
\put(7,0){$\emptyset$}
\put(13,0){$\emptyset$}
\put(19,0){$\emptyset$}
\put(25,0){$\emptyset$}
\put(31,0){$\emptyset$}
\end{picture}
\end{center}

\begin{center}
\unitlength 10pt
\begin{picture}(3,4)
\put(-0.5,1.8){$\delta$}
\put(2,2){\maru{1}}
\put(1,4){\vector(0,-1){4}}
\end{picture}
\end{center}

\begin{center}
\unitlength 10pt
\begin{picture}(34,1)
\Yboxdim{10pt}
\put(1,0){$\emptyset$}
\put(7,0){$\emptyset$}
\put(13,0){$\emptyset$}
\put(19,0){$\emptyset$}
\put(25,0){$\emptyset$}
\put(31,0){$\emptyset$}
\end{picture}
\end{center}
\medskip

\noindent
Hence this rigged configuration corresponds to the path in Example \ref{ex:typical} by $\Phi$.
\end{example}

\subsection{Inverse algorithm $\tilde{\delta}$}
For a given rigged configuration $(\tilde{\nu},\tilde{J})$ and $b\in B$ the inverse algorithm
$\tilde{\delta}$ of $\delta$ is described as follows. From $b\in B$ go back the arrow in the crystal
graph $B_0$. Let the maximal length of the singular row
in $\nu^{(a)}$ be $\tilde{\ell}_0$. Repeat the following process for $j=1,2,\ldots$ until we arrive
at $\maru{1}$. Suppose the color of the arrow is $a$. Find the maximal integer $i\le\tilde{\ell}_{j-1}$
such that $\nu^{(a)}$ has a singular row of length $i$ and set $\tilde{\ell}_j=i$, reset $b$ to be
the source of the arrow. If there are two arrows ending at $b$, compare the maximal integers and take
the larger one. If the integers the same, either one can be taken. The output of the algorithm does 
not depend on the choices.

\section{Proof of Theorem \ref{thm:bij}}
\label{sec:proof}

Theorem \ref{thm:bij} is proved in this section. The following notation
is used. Let $(\nu,J)\in\RC(\la,L)$, $b=\gamma(\nu,J)\in B$,
$\rho=\la-\wt(b)$, and $(\nut,\Jt)=\delta(\nu,J)$. 
For $(\nu,J)\in\RC(\la,L)$, define
$\Delta(c(\nu,J))=c(\nu,J)-c(\delta(\nu,J))$.
The following lemma is essentially the same as \cite[Lemma 5.1]{OSS}.

\begin{lemma} \label{lem:stat reduction} 
To prove that
\eqref{eq:c=D} holds, it suffices to show that it holds for
$L=1$, and that for $L\ge2$ with $\Phi(\nu,J)=b_1\otimes
\dotsm\otimes b_L$, we have
\begin{equation}\label{eq:Deltacc}
  \Delta(c(\nu,J)) = -\alpha^{(1)}_1,
\end{equation}
and
\begin{equation}\label{eq:Delta2cc}
H(b_{L-1}\otimes b_L)=
  \tilde{\alpha}^{(1)}_1 - \alpha^{(1)}_1
\end{equation}
where $\alpha^{(1)}_1$ and $\tilde{\alpha}^{(1)}_1$ are the lengths of the
first columns in $\nu^{(1)}$ and $\nut^{(1)}$ respectively,
and $\delta(\nu,J)=(\nut,\Jt)$.
\end{lemma}
There are five things that must be verified:
\begin{enumerate}
\item[(I)] $\rho$ is dominant.
\item[(II)] $(\nut,\Jt)\in\RC(\rho,L-1)$.
\item[(III)] $b$ can be appended to $(\nut,\Jt)$ to give $(\nu,J)$.
\item[(IV)] \eqref{eq:Deltacc} in Lemma \ref{lem:stat reduction} holds.
\item[(V)] \eqref{eq:Delta2cc} in Lemma \ref{lem:stat reduction} holds.
\end{enumerate}
Parts (I) and (II) show that $\delta$ is well-defined. Part (III) shows
$\delta$ has an inverse. Part (IV) and (V) suffice to prove that
$\Phi$ preserves statistics.

We need several preliminary lemmas on the convexity and
nonnegativity of the vacancy numbers $p^{(a)}_i$.

\begin{lemma}\label{lem:asym}
For large $i$, we have
\begin{equation*}
p_i^{(a)} = \la_a
\end{equation*}
where $\la_a$ is defined by $\la=\sum_{a\in I_0}\la_a\Lab_a$.
\end{lemma}
\begin{proof}
This follows from the formula for the vacancy number \eqref{eq:p} and 
the constraint \eqref{eq:config}.
\end{proof}

Direct calculations show that
\begin{equation}\label{eq:Pm D}
-p_{i-1}^{(a)}+2p_i^{(a)}-p_{i+1}^{(a)}
=L\delta_{a1}\delta_{i1}-2m_i^{(a)}+\sum_{b\sim a}m_i^{(b)}.
\end{equation}
In particular these equations imply the convexity condition
\begin{equation}\label{eq:convex}
p_i^{(a)}\ge \frac12(p_{i-1}^{(a)}+p_{i+1}^{(a)})
\qquad \text{if $m_i^{(a)}=0$.}
\end{equation}

\begin{lemma}\label{lem:equiv}
Let $\nu$ be a configuration. The following
are equivalent:
\begin{enumerate}
\item $p_i^{(a)}\ge 0$ for all $i\in\Z_{>0}$, $a\in I_0$;
\item $p_i^{(a)}\ge 0$ for all $i\in\Z_{>0}$, $a\in I_0$ such that
$m_i^{(a)}>0$.
\end{enumerate}
\end{lemma}
\begin{proof}
This follows immediately from Lemma \ref{lem:asym} and the convexity
condition \eqref{eq:convex}.
\end{proof}

\begin{proof}[Proof of (I)]
Here we show $\rho=\la-\wt(b)$ is dominant. Suppose not. Let $\la=\sum_{i\in I_0}
\la_i\Lab_i$. Since $\veps_i(b),\vphi_i(b)\le1$ for any $i\in I_0$ and $b\in B$, in order to
make $\rho$ not dominant there exists $a\in I_0$ such that $\la_a=0$ and $\vphi_a(b)=1$. (There may be at
most two such $a$, but the proof is uniform.)
Let $R$ be the route taken by the algorithm $\delta$. Although the arrow of color $a$ sourcing
from $b$ is not taken by $\delta$, we include it into $R$ as a terminal arrow from notational
reason. Let $(a_1,\ldots,a_l)$ be colors of arrows in $R$. Let $v_j$ be the source of the 
arrow of color $a_j$. Then $a_l=a,v_l=b$. Let $\ell_j$ be the length of the singular row in
$\nu^{(a_j)}$ whose node is removed by $\delta$. 

Let $\ell$ be the largest part in $\nu^{(a)}$. We first show $\ell>0$. Suppose
$\ell=0$. Then from \eqref{p and Q} and Lemma \ref{lem:asym} one gets
\begin{equation} \label{pr1:eq1}
0=L\delta_{a1}+\sum_{c\sim a}Q^{(c)}_i\qquad\text{for large }i.
\end{equation}
However, this is a contradiction since along the route $R$ there has
to be some $c$ such that $c\sim a$ and a node in $\nu^{(c)}$ was removed. There is
only one exception: $b=\maru{1}$ and $a=1$ case. This is also contradictory since 
the first term of the r.h.s. of \eqref{pr1:eq1} is positive. We can conclude $\ell>0$.

The convexity condition \eqref{eq:convex} implies $p^{(a)}_i=0$ for all $i\ge\ell$.
Equation \eqref{eq:Pm D} in turn yields $m^{(c)}_i=0$ for all $i>\ell$ and $c\sim a$.
Set $k=\max\{1\le j<l\mid a_j\sim a\}$. 
Then from Lemma \ref{lem:graph} (4) there is an arrow of color $a$ sourcing from $v_j$ for
any $k<j<l$, though by definition of $a_k$ and $a_l$, all these arrows are not chosen by $\delta$. 
In view of the fact that $m_i^{(a_k)}=0$ for all $i>\ell$ and $\nu^{(a)}$ has a 
singular row of length $\ell$, one concludes that all length $\ell$ rows of $\nu^{(a)}$ had been
removed before $a_k$. Thus we obtain
\begin{equation} \label{pr1:eq2}
\sharp\{1\le j<l\mid a_j=a\text{ and }\ell_j=\ell\}=m_\ell^{(a)}.
\end{equation}

Set $i=\ell$ in \eqref{eq:Pm D}. It yields
\begin{equation} \label{pr1:eq3}
-p^{(a)}_{\ell-1}=L\delta_{a1}\delta_{\ell1}-2m^{(a)}_\ell+\sum_{c\sim a}m^{(c)}_\ell.
\end{equation}
\eqref{pr1:eq2}, Lemma \ref{lem:graph} (1) imply $\sum_{c\sim a}m^{(c)}_\ell\ge2m^{(a)}_\ell$,
thus from \eqref{pr1:eq3} we deduce $p^{(a)}_{\ell-1}=0$ and
$\sum_{c\sim a}m^{(c)}_\ell=2m^{(a)}_\ell$. Let 
$l_1=\min\{1\le j<l\mid a_j=a\text{ and }\ell_j=\ell\}$. Then the latter condition combined with
Lemma \ref{lem:graph} (1) and \eqref{pr1:eq2} imply that a node in each row of length $\ell$ 
in $\nu^{(c)}$ ($c\sim a$) should be entirely removed during the process of the algorithm between 
$j=l_1$ and $j=l$. Therefore length $\ell$ rows of $\nu^{(c)}$ ($c\sim a$) are not removed between
$j=1$ and $j=l_1-1$, which implies that $\ell_j<\ell$ for all 
$j\le\max\{1\le j<l_1\mid a_j\sim a\}$. If $m^{(a)}_{\ell-1}>0$,
a node in all these rows should have been removed at the stage of $j=l_1$ during the algorithm 
since these rows are singular and after $j=l_1$ only length $\ell$ rows are removed. Hence 
\begin{equation} \label{pr1:eq4}
\sharp\{1\le j<l\mid a_j=a\text{ and }\ell_j=\ell-1\}=m^{(a)}_{\ell-1}.
\end{equation}
This equality is valid also when $m^{(a)}_{\ell-1}=0$.

Set $i=\ell-1$ in \eqref{eq:Pm D}. It yields
\begin{equation} \label{pr1:eq5}
-p^{(a)}_{\ell-2}=L\delta_{a1}\delta_{\ell-1,1}-2m^{(a)}_{\ell-1}+\sum_{c\sim a}m^{(c)}_{\ell-1}.
\end{equation}
\eqref{pr1:eq4}, Lemma \ref{lem:graph} (1) and \eqref{pr1:eq5} imply $p^{(a)}_{\ell-2}=0$ and
$\sum_{c\sim a}m^{(c)}_{\ell-1}=2m^{(a)}_{\ell-1}$. The latter condition implies $\ell_j<\ell-1$ for all
$j\le\max\{1\le j<l_2\mid a_j\sim a\}$ where $l_2=\min\{1\le j<\ell\mid a_j=a\text{ and }\ell_j\ge\ell-1\}$,
since from Lemma \ref{lem:graph} (1) a node in all the rows of length $\ell-1$ in $\nu^{(c)}$ ($c\sim a$)
should be removed between $j=l_2$ and $j=l_1$. We continue this procedure until $j=1$, where
\begin{align} 
&\sharp\{1\le j<l\mid a_j=a\text{ and }\ell_j=1\}=m^{(a)}_1, \label{pr1:eq6} \\
&-p^{(a)}_0=0=L\delta_{a1}-2m^{(a)}_1+\sum_{c\sim a}m^{(c)}_1. \label{pr1:eq7}
\end{align}
are established.

From \eqref{pr1:eq6}, Lemma \ref{lem:graph} (1) we have $\sum_{c\sim a}m^{(c)}_1\ge2m^{(a)}_1$.
It contradicts to \eqref{pr1:eq7} when $a=1$. If $a\neq1$, we have $\sum_{c\sim a}m^{(c)}_1=2m^{(a)}_1$.
This equation implies that a node in all the rows of length 1 in $\nu^{(c)}$ ($c\sim a$) should be 
removed during the process $j\ge\min\{1\le j<l\mid a_j=a\}$. However, it is a contradiction, since 
there exists a $j$ such that $a_j\sim a$ and $j<\min\{1\le j<l\mid a_j=a\}$ by Lemma \ref{lem:graph} (2).
The proof is completed.
\end{proof}

\begin{proof}[Proof of (II)]
To prove the admissibility of $(\nut,\Jt)$ we need to show 
\begin{equation} \label{pr2:eq1}
0\leq\Jt^{(a,i)}_{\max}\leq \tilde{p}_i^{(a)}
\end{equation}
for all $i\ge1,1\le a\le6$ where $\Jt^{(a,i)}_{\max}$ stands for the largest part of
$\Jt^{(a,i)}$. In view of the definition of the algorithm $\delta$ in \S\ref{subsec:algorithm}
and the tables of $\tilde{p}^{(a)}_i-p^{(a)}_i$ in \S\ref{subsec:vacancy}, the condition \eqref{pr2:eq1}
could only be violated when the following cases occur.
\begin{itemize}
\item[(i)] There exists a singular row of length $i$ in $\nu^{(a)}$ such that $\ell_j\le i<\ell_{j'}$
	for some $j<j'$.
\item[(ii)] $m^{(a)}_{\ell_{j'}-1}=0,p^{(a)}_{\ell_{j'}-1}=0,\ell_j<\ell_{j'}$ for some $j<j'$.
\end{itemize}
In both cases $\ell_{j'}$ corresponds to $\nu^{(a)}$ and $\ell_j$ to $\nu^{(c)}$ such that $c\sim a$
and $j$ is the maximum that is less than $j'$.

We show (i) and (ii) cannot occur. Firstly, suppose (i) occurs. Then, by Lemma \ref{lem:graph} (4),
a node of this singular row of length $i$ should have been removed by $\delta$, which is a contradiction.
Suppose (ii) occurs. Let $t$ be a maximal integer such that $t<\ell_{j'},m^{(a)}_t>0$; if no such $t$
exists set $t=0$. By \eqref{eq:convex} $p^{(a)}_{\ell_{j'}-1}=0$ is only possible if $p^{(a)}_i=0$ for
all $t\le i\le\ell_{j'}$. By \eqref{eq:Pm D} one finds that $m^{(c)}_i=0$ for all $c\sim a, t<i<\ell_{j'}$.
Since $\ell_j<\ell_{j'}$ this implies that $\ell_j\le t$. If $t=0$, it contradicts $\ell_j\ge1$. Hence
assume that $t>0$. Since $p^{(a)}_t=0$ and $m^{(a)}_t>0$, there is a singular row of length $t$ in 
$\nu^{(a)}$ and therefore $\ell_{j'}=t$ by Lemma \ref{lem:graph} (4), which contradicts $t<\ell_{j'}$.
\end{proof}

\begin{proof}[Proof of (III)]
Given $(\nut,\Jt)\in\Path(\rho,L-1)$ and $b\in B$, we want to show that one obtains the original 
$(\nu,J)\in\Path(\la,L)$ by the inverse procedure of $\delta$. However, once one notices from the
tables in \S\ref{subsec:vacancy} that if a node is removed from a row of length $\ell$ in $\nu^{(a)}$,
then the difference $\tilde{p}^{(a)}_i-p^{(a)}_i=+1$ for all $\ell\le i<\ell'$ where $\ell'$ is the
length of the singular row in $\nu^{(c)}$ such that $c\sim a$ removed by $\delta$ after $\ell$, it is 
obvious that $\tilde{\delta}$ gives the inverse procedure of $\delta$.
\end{proof}

\begin{proof}[Proof of (IV)]
Let $(\nut,\Jt)=\delta(\nu,J)$. Let $\tilde{m}^{(a)}_i,\tilde{p}^{(a)}_i$ be for $(\nut,\Jt)$. Let
$\ell^{(a)}_k$ ($1\le k\le k_a$) be the length of the row a node of which is removed at the $k$-th
time from $\nu^{(a)}$ by the algorithm $\delta$. Then by \eqref{eq:c},\eqref{L^(a)_i},\eqref{c-nu-J}
we have
\begin{align} \label{pr4:eq1}
\Delta(c(\nu,J))=&\frac12\sum_{a,b}\sum_{j,k}C_{ab}\min(j,k)
(m^{(a)}_jm^{(b)}_k-\tilde{m}^{(a)}_j\tilde{m}^{(b)}_k)\\
&+\sum_j(Lm^{(1)}_j-(L-1)\tilde{m}^{(1)}_j)+\sum_a\sum_{k=1}^{k_a}
(p^{(a)}_{\ell^{(a)}_k}-\tilde{p}^{(a)}_{\ell^{(a)}_k-1}). \nonumber
\end{align}
From \eqref{eq:p} we obtain
\begin{align*}
p^{(a)}_{\ell^{(a)}_k}-\tilde{p}^{(a)}_{\ell^{(a)}_k-1}
&=\delta_{a1}(1+(L-1)\delta_{\ell^{(a)}_k,1})\\
&-\sum_{b,j}C_{ab}\Bigl(\chi(j\ge\ell^{(a)}_k)m^{(b)}_j+\min(\ell^{(a)}_k-1,j)
\sum_{i=1}^{k_b}(\delta_{j,\ell^{(b)}_i}-\delta_{j,\ell^{(b)}_i-1})\Bigr).
\end{align*}
Substituting \eqref{eq:ch m D1} and the above into \eqref{pr4:eq1} one gets
\[
\Delta(c(\nu,J))=k_1-\sum_jm^{(1)}_j-V,
\]
where
\[
V=\frac12\sum_{a,b}\sum_{i=1}^{k_a}\sum_{j=1}^{k_b}C_{ab}(\delta_{\ell^{(a)}_i\ell^{(b)}_j}
+\chi(\ell^{(a)}_i<\ell^{(b)}_j)).
\]
Use another notation for $\ell^{(a)}_i$. Namely, let $\ell_j$ ($j=1,\ldots,\ell$) be the successive
length of the singular rows by $\delta$. $V$ is calculated as
\begin{align*}
V&=\frac12\sum_{i,j=1}^\ell C_{a_ia_j}(\delta_{\ell_i\ell_j}+2\chi(\ell_i<\ell_j))\\
&=\ell+\sum_{i<j}C_{a_ia_j}\\
&=k_1.
\end{align*}
Here we have used Lemma \ref{lem:graph} (2) in the last equality. This completes the proof.
\end{proof}

\begin{proof}[Proof of (V)]
The proof is reduced to showing the following lemma.
\end{proof}

\begin{lemma}
For $(\nu,J)\in\RC(\la,L)$ with $L\ge2$ set $\gamma(\nu,J)=c,\gamma(\delta(\nu,J))=b$. Let 
$\ell^{(a)}_k$ be the length of the singular row in $\nu^{(a)}$ at the $k$-th time by the algorithm
$\delta$. Define the following subsets of $B\otimes B$.
\begin{align*}
S_1=&\{\maru{1}\otimes\maru{j}\mid j\ge18\}\sqcup\{\maru{2}\otimes\maru{j}\mid j\ge23\}\sqcup
\{\maru{3}\otimes\maru{j}\mid j\ge25\}\\
&\sqcup\{\maru{4}\otimes\maru{j},\maru{7}\otimes\maru{j}\mid j\ge26\}\sqcup
\{\maru{i}\otimes\maru{27}\mid i=5,8,10,13,18\},\\
S_2=&\{\maru{i}\otimes\maru{j}\mid \maru{i}\text{ can be reached by following some (possibly zero) 
arrows from }\maru{j}\}.
\end{align*}
Then we have
\begin{enumerate}
\item $H(b\otimes c)=\left\{
\begin{array}{ll}
-2& \text{if }b\otimes c\in S_1\\
0& \text{if }b\otimes c\in S_2\\
-1& \text{otherwise}.
\end{array}\right.$
\item $b\otimes c$ belongs to $S_1$ if and only if $\ell^{(1)}_1=\ell^{(1)}_2=1$.
\item $b\otimes c$ belongs to $S_2$ if and only if $\ell^{(1)}_1>1$.
\end{enumerate}
\end{lemma}

\begin{proof}
Checking (1) reduces to a finite calculation that can be confirmed by computer.

To prove (2) let $\tilde{\ell}^{(a)}_k$ be the length of the row in $\nu^{(a)}$ at the $k$-th time
by the second $\delta$. We first show the condition $\ell^{(1)}_1=\ell^{(1)}_2=1$ is equivalent to 
\begin{equation} \label{pr:lemH}
\ell^{(2)}_3\le\tilde{\ell}^{(1)}_1,\ell^{(3)}_4\le\tilde{\ell}^{(2)}_1,
\ell^{(4)}_3\le\tilde{\ell}^{(3)}_1,\ell^{(5)}_2\le\tilde{\ell}^{(4)}_1,
\tilde{\ell}^{(5)}_1=\infty.
\end{equation}
Let $R$ and $\tilde{R}$ be the routes taken by the first and second algorithms $\delta$. Suppose
$\ell^{(1)}_1=\ell^{(1)}_2=1$. Then for all the arrows in $R$ between the first one of color 1 
and the second, the first $\delta$ removes a node from a row of length 1, namely, removes the row.
In view of the table for $a=1$ in \S\ref{subsec:vacancy} the length of the singular row after the
first $\delta$ should be no less than $\ell^{(2)}_3$. Hence we have $\ell^{(2)}_3\le
\tilde{\ell}^{(1)}_1$. For the next inequality view the table for $a=2$. Since $\ell^{(2)}_3\le
\tilde{\ell}^{(2)}_1$, we get $\ell^{(3)}_4\le\tilde{\ell}^{(2)}_1$. Proceeding similarly we obtain
\eqref{pr:lemH}. Suppose \eqref{pr:lemH} next and assume $\ell^{(1)}_2>1$. Then after the first 
$\delta$ there exists a singular row in $\nu^{(1)}$ of length less than $\ell^{(1)}_2$, which means
$\tilde{\ell}^{(1)}_1<\ell^{(1)}_2$. However, it contradicts to the first inequality of 
\eqref{pr:lemH}. Therefore, we have $\ell^{(1)}_1=\ell^{(1)}_2=1$. The fact that \eqref{pr:lemH} is
equivalent to $b\ot c\in S_1$ is checked as follows. Suppose for instance that $b=\maru{2}$.
This means $\tilde{\ell}^{(1)}_1<\infty$ and $\tilde{\ell}^{(2)}_1=\infty$. From the first inequality
of \eqref{pr:lemH} we have $\ell^{(2)}_3<\infty$, which implies $c=\maru{j}$ for $j\ge23$. Other
cases can be checked similarly.

We are left to show (3). From the assumption $\ell^{(1)}_1>1$, there are remaining singular rows
after the first $\delta$ which could be removed by the second $\delta$. Thus the ``if" part is 
finished. To show the ``only if" part, we assume $\ell^{(1)}_1=1$ and deduce a contradiction. 
From the table at \S\ref{subsec:vacancy}, $a=1$, the value for $[\ell^{(1)}_1,\ell^{(2)}_1)$ is
$+1$ while we have $\ell^{(1)}_1=1$. Thus we have $\tilde{\ell}^{(1)}_1\ge\ell^{(2)}_1$. In view
of the table at \S\ref{subsec:vacancy}, at $a=2$, the value for $[\ell^{(2)}_1,\ell^{(3)}_1)$ is
$+1$. Thus we find $\tilde{\ell}^{(2)}_1\ge\ell^{(3)}_1$. We can continue this procedure
as follows. For $\ell^{(a)}_k<\infty$ one can definitely find $\tilde{\ell}^{(a)}_k<\infty$ by the
assumption $b\otimes c\in S_2$. Imitating the way to show $\tilde{\ell}^{(1)}_1\ge\ell^{(2)}_1$
and $\tilde{\ell}^{(2)}_1\ge\ell^{(3)}_1$, we can then find a pair $(a',k')$ such that 
$\ell^{(a)}_k\le\ell^{(a')}_{k'}$ and $\tilde{\ell}^{(a)}_k\ge\ell^{(a')}_{k'}$.
This procedure continues until we arrive at $\ell^{(a')}_{k'}=\infty$. However, the previous 
$\tilde{\ell}^{(a)}_k$ should be finite since the second $\delta$ can go further along the route
taken by the first $\delta$. This contradicts to $\tilde{\ell}^{(a)}_k\ge\ell^{(a')}_{k'}=\infty$.
The proof is finished.
\end{proof}

\subsection*{Acknowledgements}
M.O. thanks Katsuyuki Naoi for stimulating discussions. M.O. is partially supported by the Grants-in-Aid 
for Scientific Research No. 20540016 from JSPS.

\end{document}